\documentclass[12pt]{amsart}
\usepackage{graphicx}
\numberwithin{equation}{section}

\newcommand{\ben}{\begin{enumerate}}
\newcommand{\een}{\end{enumerate}}

\newcommand{\bea}{\begin{eqnarray}}
\newcommand{\ba}{\begin{array}}
\newcommand{\bean}{\begin{eqnarray*}}
\newcommand{\ea}{\end{array}}
\newcommand{\eea}{\end{eqnarray}}
\newcommand{\eean}{\end{eqnarray*}}
\newcommand{\beq}{\begin{equation}}
\newcommand{\eeq}{\end{equation}}
\newcommand{\bthm}{\begin{thm}}
\newcommand{\ethm}{\end{thm}}
\newcommand{\blem}{\begin{lem}}
\newcommand{\elem}{\end{lem}}
\newcommand{\bprop}{\begin{prop}}
\newcommand{\eprop}{\end{prop}}
\newcommand{\bcor}{\begin{cor}}
\newcommand{\ecor}{\end{cor}}
\newcommand{\bdfn}{\begin{dfn}}
\newcommand{\edfn}{\end{dfn}}
\newcommand{\brem}{\begin{rem}}
\newcommand{\erem}{\end{rem}}
\newcommand{\bpf}{\begin{proof}}
\newcommand{\epf}{\end{proof}}
\newcommand{\bfact}{\begin{fact}}
\newcommand{\efact}{\end{fact}}

\alph{enumii} \roman{enumiii}

\newtheorem{thm}{Theorem}[section]
\newtheorem{prop}[thm]{Proposition}
\newtheorem{lem}[thm]{Lemma}

\newtheorem{cor}[thm]{Corollary}
\newtheorem{dfn}[thm]{Definition}
\newtheorem{rem}[thm]{Remark}
\newtheorem{fact}[thm]{Fact}
\newtheorem{ex}[thm]{Example}

\def\N{{\mathbb N}}                \def\Z{{\mathbb Z}}      \def\R{{\mathbb R}}
\def\C{{\mathbb C}}                  \def\oc{\hat \C}

\def\1{1\!\!1}
\def\and{\text{ and }}

      \def\Exp{\text{{\rm Exp}}}
\def\Epb{\text{{\rm Epb}}}

     \def\HD{\text{{\rm HD}}}

              \def\om{\omega}           
\def\Sg{\Sigma}               \def\sg{\sigma}

\def\({\bigl(}                \def\){\bigr)}

                        \def\^{\tilde}

\def\es{\emptyset}

\def\sp{\medskip}

\def\om{\omega}

\begin{document}

\title[]
{\bf\large {\Large B}owen  Parameter and Hausdorff Dimension for Expanding 
Rational Semigroups}
\date{February 8, 2012. Published in Discrete and Continuous Dynamical Systems Ser. A, 
 32 (2012), no. 7, 2591--2606.}
\author[\sc Hiroki SUMI]{\sc Hiroki SUMI}
%
\author[\sc Mariusz URBA\'NSKI]{\sc Mariusz URBA\'NSKI}
%
%
\thanks{The first author thanks University of North Texas for support and kind hospitality. 
Research of the first author was partially supported by JSPS KAKENHI 21540216.  
Research of the second author supported in part by the
NSF Grant DMS 0400481.}
\thanks{\ \newline 
\noindent Hiroki Sumi\newline 
Department of Mathematics,
Graduate School of Science,
Osaka University, 
1-1 Machikaneyama,
Toyonaka,
Osaka, 560-0043, 
Japan\newline 
E-mail: sumi@math.sci.osaka-u.ac.jp\newline 
Web: http://www.math.sci.osaka-u.ac.jp/$\sim$sumi/\newline
\ \newline 
Mariusz Urba\'nski\newline Department of Mathematics,
 University of North Texas, Denton, TX 76203-1430, USA\newline  
E-mail: urbanski@unt.edu\newline
Web: http://www.math.unt.edu/$\sim$urbanski/}

\begin{abstract}
We estimate the Bowen parameters and the Hausdorff dimensions of the
Julia sets of expanding finitely generated rational semigroups.  
We show that the Bowen parameter is larger than or equal to the ratio of the
entropy of the skew product map $\tilde{f}$ and the Lyapunov exponent of $\tilde{f}$ with 
respect to the maximal entropy measure for $\tilde{f}$. Moreover, we show that the equality holds if and only if 
the generators are simultaneously conjugate to the form $a_{j}z^{\pm d}$ by a M\"{o}bius transformation.  
Furthermore, we show that there are plenty of expanding finitely generated rational semigroups such that 
the Bowen parameter is strictly larger than $2$. 
\end{abstract}
\maketitle

\section{Introduction}
\label{Introduction}
A {\bf rational semigroup} 
is a semigroup generated by a family of 
non-constant rational maps $g:\oc \rightarrow \oc $, 
where $\oc $ denotes the Riemann sphere, 
with the semigroup operation being functional composition. 
A polynomial semigroup is a semigroup generated by a 
family of non-constant polynomial maps on $\oc .$ 
The work on the dynamics of rational semigroups was initiated 
by A. Hinkkanen and G. J. Martin (\cite{HM}), 
who were interested in the role of the dynamics of polynomial semigroups 
while studying various one-complex-dimensional moduli spaces for discrete 
groups of M\"{o}bius transformations, and by F. Ren's group 
(\cite{ZR}), who studied such semigroups from the perspective 
of random dynamical systems. 
  
 The theory of the dynamics of rational semigroups on $\oc $ 
has developed in many directions since the 1990s (\cite{HM, ZR,
   SSS, sumihyp1, sumihyp2,  
hiroki1, sumi1, sumi2, sumi06, sumirandom,
su0, su1, sumid1, sumid3, SS, sumiintcoh, sumiprepare}).  
Since the Julia set $J(G)$ of a rational semigroup 
generated by finitely many elements $f_{1},\ldots ,f_{s}$ 
has {\bf backward self-similarity} i.e. 
\begin{equation}
\label{bsseq}
 J(G)=f_{1}^{-1}(J(G))\cup \cdots \cup f_{s}^{-1}(J(G)),
\end{equation}  
(see \cite{sumihyp1, hiroki1}), 
it can be viewed as a significant generalization and extension of 
both the theory of iteration of rational maps (see \cite{M}) 
and conformal 
iterated function systems (see \cite{MU}). 
Indeed, because of (\ref{bsseq}), 
the analysis of the Julia sets of rational semigroups somewhat
resembles 
 ``backward iterated functions systems'', however since each map 
$f_{j}$ is not in general injective (critical points), some 
qualitatively different extra effort in the cases of semigroups is needed.
The theory of the dynamics of 
rational semigroups borrows and develops tools 
from both of these theories. It has also developed its own 
unique methods, notably the skew product approach 
(see \cite{hiroki1, sumi1, sumi2, sumi06, sumid1, sumid2, sumid3, sumiprepare,su0,suetds1,su1}).

The theory of the dynamics of rational semigroups is intimately 
related to that of the random dynamics of rational maps. 
For the study of random complex dynamics, the reader may 
consult \cite{FS,Bu1,Bu2,BBR,Br,GQL, sumiprepare}. The deep relation between these fields 
(rational semigroups, random complex dynamics, and (backward) IFS) 
is explained in detail in the subsequent papers
 (\cite{sumirandom, sumid1, sumid2, sumid3,  sumikokyuroku, sumiintcoh, sumiprepare}) of the first author.

 In this paper, we deal at length with  Bowen's parameter $\delta $ (the unique zero of the pressure function) 
of expanding finitely generated rational semigroups $\langle f_{1},\ldots ,f_{s}\rangle $ (see Definition~\ref{d:Bpara}). 
In the usual iteration dynamics of a single expanding rational map, 
it is well known that the Hausdorff dimension of the Julia set is equal to 
the Bowen's parameter. For a general expanding finitely generated rational semigroup 
$\langle f_{1},\ldots ,f_{s}\rangle $,  
it was shown that the Bowen's parameter is larger than or equal to the
Hausdorff dimension of the Julia set  
(\cite{sumihyp2,sumi2}). If we assume further that the semigroup satisfies
the ``open set condition'' (see Definition~\ref{d:osc}),  
then it was shown that they are equal (\cite{sumi2}). 
However, if we do not assume the open set condition, then 
there are a lot of examples such that the Bowen's parameter is strictly larger than 
the Hausdorff dimension of the Julia set. In fact, 
the Bowen's parameter can be  strictly larger than two. 
Thus, it is very natural to ask when we have this situation and what
happens if we have such a case.  
We will show the following.
\begin{thm}[Theorem~\ref{t:main1}]
\label{t:main1int}
For an expanding rational semigroup $G=\langle
f_{1},\ldots ,f_{m}\rangle $,  
the Bowen's parameter $\delta $ satisfies 
\begin{equation}
\label{eq:mainineq}
\delta \geq \frac{\log (\sum _{j=1}^{s}\deg (f_{j}))}{\int \log \| \tilde{f}'\| d\mu },
\end{equation}  
where $\tilde{f}$ denotes the skew product map associated with 
the multi-map $f=(f_{1},\ldots ,f_{s})$ (see section~\ref{Prel}), 
and $\mu $ denotes the unique
maximal entropy measure for $\tilde{f}$ (see \cite{pubook,hiroki1}).    
Moreover,  the equality in the \eqref{eq:mainineq} holds  if and only if 
we have a very special condition, i.e., there exists a M\"{o}bius
transformation $\varphi $  
and a positive integer $d_{0}$ 
such that for each $j$, $\varphi f_{j}\varphi ^{-1}(z)$ is of the form $a_{j}z^{\pm d_{0}}$.    
\end{thm} 
Note that 
$\log (\sum _{j=1}^{s}\deg (f_{j}))$ is equal to the entropy of $\tilde{f}.$ 
The above result (Theorem~\ref{t:main1}) generalizes a weak form of
A. Zdunik's theorem (\cite{zdunik2}), which is a result  
for the usual iteration of a single rational map. In fact, in the
proof of the main result of our paper,  
Zdunik's theorem is one of the key ingredients. 
We emphasize that in the main
result of our paper, we can take the M\"{o}bius map $\varphi $ which does not depend on $j.$   
    
If each $f_{j}$ is a polynomial with $\deg (f_{j})\geq 2$, then by using potential theory, we can calculate $\int \log \| \tilde{f}'\| d\mu $ 
in (\ref{eq:mainineq}) 
in terms of $\deg (f_{j})$ and an integral related to fiberwise Green's functions 
 (see Lemmas~\ref{l:lyapcalc}, \ref{l:lyappbd}). 
From this calculation, we can prove the following.  
\begin{thm}[Theorem~\ref{t:main2}]
\label{t:main2int}
Let $s\in \Bbb{N} $ and 
for each $j=1,\ldots ,s$, let $f_{j}$ be a polynomial with $\deg (f_{j})\geq 2.$ 
If $G=\langle f_{1},\ldots ,f_{s}\rangle $ is an expanding polynomial semigroup, 
the postcritical set of $G$ in $\C $ is bounded, 
$(\log d)/(\sum _{j=1}^{s}\frac{d_{j}}{d}\log d_{j})\geq 2$ where
$d_{j}:=\deg (f_{j})$ and $d=\sum _{j=1}^{s}d_{j}$,  
and $\delta \leq 2$, then there exists a M\"{o}bius transformation
$\varphi $ such that for each $j$,  
$\varphi f_{j}\varphi ^{-1}(z)$ is of the form $a_{j}z^{s}.$ 
\end{thm} 
Thus, 
if the postcritical set of $G$ in $\C $ is bounded and $(\log d)/(\sum
_{j=1}^{s}\frac{d_{j}}{d}\log d_{j})\geq 2$,  
then typically we have that $\delta >2.$ Note that in the usual iteration
dynamics of a single rational map, we always have $\delta \leq 2.$ 

Therefore, we can say that there are plenty of expanding finitely
generated polynomial semigroups for which   
the Bowen's parameter is strictly larger than 2. 

Note that combining these estimates of Bowen's parameter and the
``transversal family'' type arguments, we will show 
that we have a large amount of expanding $2$-generator polynomial semigroups $G$ 
such that the Julia set of $G$ has positive 2-dimensional Lebesgue measure (\cite{su}).

 We remark that, as illustrated in \cite{sumikokyuroku, sumiprepare}, 
 estimating the Hausdorff dimension of the Julia sets of 
 rational semigroups plays an important role when we 
 investigate random complex dynamics and its associated 
 Markov process on $\oc .$ For more details, see Remark~\ref{r:dimran} and 
\cite{sumikokyuroku, sumiprepare}.  
\section{Preliminaries}
\label{Prel} 
In this section we introduce notation and basic definitions. 
Throughout the paper, we frequently follow the notation 
from \cite{hiroki1} and \cite{sumi2}. 
\begin{dfn}[\cite{HM,ZR}] 
A ``rational semigroup" $G$ is a semigroup generated by a family of 
non-constant 
rational maps $g:\oc \rightarrow \oc$,\ where $\oc $ denotes the 
Riemann sphere,\ with the semigroup operation being functional 
composition. 
A ``polynomial semigroup'' is a semigroup generated by a 
family of non-constant polynomial maps on $\oc .$ 
For a rational semigroup $G$, we set 
$$
F(G):=\{ z\in \oc \mid G \mbox{ is normal in a neighborhood of } z\} 
$$
and we call $F(G)$ the {\bf Fatou set} of $G$. Its complement,
$$
J(G):=\oc \setminus F(G)
$$ 
is called the {\bf Julia set} of $G.$ 
If $G$ is generated by a family $\{ f_{i}\} _{i}$,\ 
then we write $G=\langle f_{1},f_{2},\ldots \rangle .$ 
\end{dfn} 
For the papers dealing with dynamics of rational semigroups, 
see for example \cite{HM,ZR,SSS, sumihyp1,sumihyp2,
hiroki1,sumi1, sumi2,sumi06, sumirandom, su0, su1, 
sumid1, sumid2, sumid3,SS,sumiintcoh, sumiprepare, sumikokyuroku}, etc. 

We denote by  Rat the set of all non-constant 
rational maps on $\oc $ endowed with 
the topology induced by uniform convergence on $\oc .$ 
Note that  Rat has countably many connected 
components. In addition, each connected component 
$U$ of Rat is an open subset of Rat and 
$U$ has a structure of a finite dimensional complex manifold.  
Similarly, we denote by ${\mathcal P}$ the set of all polynomial maps
$g:\oc \rightarrow \oc $  
with $\deg (g)\geq 2$ endowed with the relative topology from Rat. 
Note that ${\mathcal P}$ has countably many connected 
components. In addition, each connected component 
$U$ of ${\mathcal P}$ is an open subset of ${\mathcal P}$ and 
$U$ has a structure of a finite dimensional complex manifold.

Let $V$ be an open subset of $\oc $ and let $z\in V.$ 
 For a holomorphic map $g:V\rightarrow \oc $, 
we denote by $\| g'(z)\| $ the norm of the derivative of $g:V\rightarrow \oc $
 at $z$ with respect to the spherical metric on $\oc .$ 

\begin{dfn}
For each $s\in \Bbb{N}$, 
let $\Sigma _{s}:=\{ 1,\ldots ,s\} ^{\Bbb{N}}$ be the 
space of one-sided sequences of $s$-symbols endowed with the 
product topology. This is a compact metrizable space. 
For each $f=(f_{1},\ldots ,f_{s})\in (\mbox{{\em Rat}})^{s}$, 
we define a map   
$$
\tilde{f}:\Sg_{s}\times \oc \rightarrow \Sg_{s}\times \oc 
$$ 
by the formula 
$$
\tilde{f}(\om,z)=(\sg (\om ),\ f_{\om_{1}}(z)),
$$
where $(\om,z)\in \Sg _{s}\times \oc,\ \om=(\om_{1},\om_{2},\ldots ),$ and 
$\sg :\Sigma _{s}\rightarrow \Sg _{s}$ denotes the shift map. 
The transformation $\tilde{f} :\Sigma _{s}\times \oc \rightarrow 
\Sigma _{s}\times \oc $ is called the {\bf skew product map} associated 
with the multi-map $f=(f_{1},\ldots ,f_{s})\in (\mbox{{\em Rat}})^{s} .$  
We denote by $\pi _{1}:\Sigma _{s}\times \oc \rightarrow \Sigma _{s}$ 
the projection onto $\Sg_{s}$ and by $\pi_{2}:\Sg _{s}\times\oc\rightarrow\oc$ 
the projection onto $\oc $. That is, $\pi _{1}(\om ,z)=\om $ and 
$\pi _{2}(\om ,z)=z.$  For each $n\in \Bbb{N} $ and $(\om ,z)\in 
\Sigma _{s}\times \oc $, we put 
$$
(\tilde{f}^{n})'(\om ,z):= (f_{\om_{n}}\circ \cdots \circ f_{\om _{1}})'(z).
$$ 
Moreover, we denote by $\| (\tilde{f}^{n})'(\om ,z)\| $ the norm of the derivative 
of $f_{\om _{n}}\circ \cdots \circ f_{\om _{1}}$ at $z$ with respect to the 
spherical metric on $\oc .$ 
We define 
$$J_{\om }(\tilde{f}):=\{ z\in \oc \mid 
 \{ f_{\om _{n}}\circ \cdots \circ f_{\om _{1}}\} _{n\in \Bbb{N}} \mbox{ is 
 not normal in each neighborhood of } z\} 
$$ 
for each $\om \in  \Sigma _{s}$ and we set 
$$
J(\tilde{f}):= \overline{\cup _{w\in \Sigma _{s}}\{ \om \} 
\times J_{\om }(\tilde{f}) },
$$ 
where the closure is taken with respect to the product topology on the space
$\Sigma _{s}\times \oc .$ $J(\tilde{f})$ is called the 
{\bf Julia set} of the skew product map $\tilde{f}.$ In addition, we set 
$F(\tilde{f}):=(\Sigma _{s}\times \oc )\setminus J(\tilde{f}).$ and 
$
\deg(\tilde{f}):=\sum _{j=1}^{s}\deg(f_{j}).
$
\end{dfn}

\begin{rem}
\label{rem1} 
By definition, the set
 $J(\tilde{f})$ 
 is compact. Furthermore,\ if we set 
 $G=\langle f_{1},\ldots ,f_{s}\rangle $, then, 
 by  \cite[Proposition 3.2]{hiroki1},\ the following hold:  
\begin{enumerate}
\item  
 $J(\tilde{f})$ is completely invariant under $\tilde{f}$;\ 
\item 
 $\tilde{f}$ is an open map on $J(\tilde{f})$;\ 
\item 
if 
 $\sharp J(G)\geq 3$ and $E(G):= \{ z\in \oc \mid 
 \sharp \cup _{g\in G}g^{-1}\{ z\}<\infty \} $ is contained in 
 $F(G)$,  then  the dynamical system $(\tilde{f},J(\tilde{f}))$ is topologically exact; 
\item  
 $J(\tilde{f})$ is equal to the closure of 
 the set of repelling periodic points of 
 $\tilde{f}$ if 
 $\sharp J(G)\geq 3$,\ where we say that a periodic point 
 $(\om ,z)$ of $\tilde{f}$ with 
 period $n$ is repelling if $\| (\tilde{f}^{n})'(\om ,z)\| >1$.    
\item $\pi _{2}(J(\tilde{f}))=J(G).$ 
\end{enumerate}
\end{rem} 
\begin{dfn}[\cite{sumi2}]
A finitely generated 
rational semigroup $G=\langle f_{1},\ldots ,f_{s}\rangle $ 
is said to be expanding provided that 
$J(G)\ne \es $ and the skew product map 
$\tilde{f}:\Sg_{s} \times \oc \rightarrow \Sg _{s}\times \oc $ 
associated with 
$f=(f_{1},\ldots ,f_{s}) $ is expanding along 
fibers of the Julia set $J(\tilde{f})$, 
meaning that there exist $\eta >1$ and 
$C\in(0,1]$ such that for all $n\ge 1$,
\begin{equation}
\label{1112505}
\inf \{ \| (\tilde{f}^{n})'(z)\|: z\in J(\tilde{f})\} \ge C\eta ^{n}.   
\end{equation}  
\end{dfn} 
\begin{dfn}
Let $G$ be a rational semigroup. 
We put
$$
P(G):=\overline{\cup _{g\in G}\{ \mbox{all critical values of }
g:\oc \rightarrow \oc \}} \ (\subset \oc )
$$ 
and we call $P(G)$ the {\bf postcritical set} of $G$. 
A rational semigroup $G$ is said to be {\bf hyperbolic} if 
$P(G)\subset F(G).$ 
\end{dfn}
\begin{dfn}
Let $G$ be a polynomial semigroup. 
We set $P^{\ast }(G):= P(G)\setminus \{ \infty \} .$ 
We say that $G$ is postcritically bounded if $P^{\ast }(G)$ is bounded in $\C .$ 
\end{dfn}
\begin{rem}
\label{exphyplem}
Let $G=\langle f_{1},\ldots ,f_{s}\rangle $ be a rational semigroup 
such that 
there exists an element $g\in G$ with $\deg (g)\geq 2$ and 
such that each M\"{o}bius transformation in $G$ is loxodromic. 
Then,  
it was proved in \cite{sumihyp2} that 
$G$ is expanding if and only if $G$ is hyperbolic. 
\end{rem}
\begin{dfn}
We define
$$\Exp (s):=\{ (f_{1},\ldots ,f_{s})\in (\mbox{{\em Rat}})^{s}
\mid \langle f_{1},\ldots ,f_{s}\rangle \mbox{ is expanding} \}.
$$  
We also set $\Sigma _{s}^{\ast }:= \cup _{j=1}^{\infty }
\{ 1,\ldots ,s\} ^{j}$ (disjoint union). 
For every 
$\om\in\Sg_s\cup\Sg_s^*$ let $|\om|$ be the length of $\om .$
For each $f=(f_{1},\ldots ,f_{s})\in (\mbox{{\em Rat}})^{s}$ 
and each $\om =(\om _{1},\ldots ,\om _{n})\in \Sigma _{s}^{\ast }$,  
we put $f_{\om }:= f_{\om _{n}}\circ \cdots \circ f_{\om _{1}}.$ 
\end{dfn}
Then we have the following.
\begin{lem}[\cite{sumihyp1, suetds1}]
\label{expopenlem}
$\Exp(s)$ is an open subset of {\em (Rat}$)^{s}.$ 
\end{lem}
\begin{dfn}
We set 
$$\Epb (s):= \{ f=(f_{1},\ldots ,f_{s})\in \Exp (s)\cap {\mathcal P}^{s}
\mid \langle f_{1},\ldots ,f_{s}\rangle \mbox{ is postcritically bounded} \} ,$$ 
\end{dfn}
\begin{lem}[\cite{sumid3, sumiprepare}]
\label{l:epbopen}
$\Epb (s)$ is open in ${\mathcal P}^{s}.$ 
\end{lem}
\begin{dfn}
\label{d:Bpara}
Let $f=(f_{1},\ldots ,f_{s})\in \Exp(s)$ and 
let $\tilde{f}:\Sigma _{s}\times \oc \rightarrow 
\Sigma _{s}\times \oc $ be the skew product map 
associated with $f=(f_{1},\ldots ,f_{s}) .$ 
For each $t\in \R $, 
let $P(t,f)$ be the topological pressure of 
the potential $\varphi(z):= -t\log \| \tilde{f}'(z)\|$ 
with respect to the map 
$\tilde{f}:J(\tilde{f})\rightarrow J(\tilde{f}).$  
(For the definition of the topological pressure, 
see \cite{pubook}.) 
We denote by $\delta (f)$ the unique zero 
of $t\mapsto P(t,f).$ (Note that the existence and 
the uniqueness of the zero of $P(t,f)$ was shown in 
\cite{sumi2}.) The number $\delta (f)$ is called the 
{\bf Bowen parameter} of the semigroup $f=(f_{1},\ldots ,f_{s})\in 
\Exp(s).$  
\end{dfn}

We have the following  fact, which is one of the main results of \cite{suetds1}. 

\begin{thm}[\cite{suetds1}]
The function $\Exp(s)\ni f\mapsto \delta (f)\in\R$ is real-analytic and plurisubharmonic.  
\end{thm} 
\begin{dfn}
For a subset $A$ of $\oc $, we denote by $\HD (A)$ the Hausdorff dimension of 
$A$ with respect to the spherical metric. 
For a Riemann surface $S$, we denote by $\mbox{{\em Aut}}(S)$ the set of all 
holomorphic isomorphisms of $S.$ 
For a compact metric space $X$, we denote by $C(X)$ the space of all continuous 
complex-valued functions on $X$, endowed with the supremum norm. 
\end{dfn}

\section{Results}
In this section, we prove our main results.
Note that for any $f=(f_{1},\ldots ,f_{s})\in \Exp(s)$, by
Remark~\ref{rem1}, \cite{pubook}, and \cite{hiroki1},   
there exists a unique maximal entropy measure $\mu $ for
$\tilde{f}:J(\tilde{f})\rightarrow J(\tilde{f})$ 
and $h_{\mu }(\tilde{f})=h(\tilde{f})=\log (\deg (\tilde{f})).$ 
 We start with the following.
\begin{thm}
\label{t:main1} 
Let $f=(f_{1},\ldots ,f_{s})\in \Exp(s).$
Let $d:=\deg (\tilde{f})$ and
let $d_{j}=\deg (f_{j})$ for each $j=1,\ldots ,s.$ 
Let $\mu $ be the maximal entropy measure for
$\tilde{f}:J(\tilde{f})\rightarrow J(\tilde{f})$.    
Then the following statements {\em (1)} and {\em (2)} hold.  
\begin{itemize}
\item[(1)] 
$$\delta (f)\geq \frac{\log d}{\int _{J(\tilde{f})} \log \| \tilde{f}'\| d\mu }.$$ 
\item[(2)] 
Suppose $d_{1}\geq 2.$  
If 
$$\delta (f)= \frac{\log d}{\int _{J(\tilde{f})} \log \| \tilde{f}'\| d\mu },$$
then, the following items {\em (a),(b),(c)} hold.
\begin{itemize}
\item[(a)] 
\sp $d_{1}=\cdots =d_{s}.$ We set $d_{0}:=d_{1}=\cdots =d_{s}.$ 
\item[(b)] 
\sp There exist an automorphism $\varphi \in \mbox{{\em Aut}}(\oc )$ and complex
numbers $a_{1},\ldots , a_{s}$ with $a_{1}=1$  
such that for each $j=1,\ldots ,s$,  
$$
\varphi f_{j}\varphi ^{-1}(z)=a_{j}z^{\pm d_{0}}. 
$$
\item[(c)] 
\sp $\delta (f)=1+\frac{\log s}{\log d_{0}}.$ 
\end{itemize}
\end{itemize} 
\end{thm}
\begin{proof} 
We have that 
$\R\ni t\mapsto P(t,f)$ is convex and real-analytic (\cite{sumi2}, \cite{suetds1}). 
Also,
$$
\frac{\partial P(t,f)}{\partial t}|_{t=0}=-\int _{J(\tilde{f})}\log \|
\tilde{f}'\| d\mu.
$$ 
From the convexity of $P(t,f)$, we obtain that
$$
\delta(f) \geq  \frac{\log d}{\int _{J(\tilde{f})} \log \|
  \tilde{f}'\| d\mu }.
$$
We now assume that $d_{1}\geq 2$ and 
$$\
\delta (f)= \frac{\log d}{\int _{J(\tilde{f})} \log \| \tilde{f}'\|
  d\mu }. 
$$ 
Because of the convexity of $P(t,f)$ again, we infer that
$$
\frac{\partial P(t,f)}{\partial t}=-\int _{J(\tilde{f})}\log \|
\tilde{f}'\| d\mu 
$$ 
for all $t\in \R .$ 
Let $\nu $ be the unique $\delta (f)$-conformal measure on $J(\tilde{f})$ for 
$(\tilde{f}, J(\tilde{f}))$ (see \cite{sumi2}). 
Let 
$$
L_{\nu }: C(J(\tilde{f}))\rightarrow C(J(\tilde{f}))
$$ 
be the operator, called the transfer operator, defined by the following formula
$$
L_{\nu }(\varphi )(z)=\sum _{\tilde{f}(y)=z}\varphi (y)\|
\tilde{f}'(y)\| ^{-\delta (f)}.
$$
In virtue of \cite{sumi2}, the limit $\alpha := \lim _{l\rightarrow \infty
}L_{\nu }^{l}(1)\in C(J(\tilde{f}))$ exists,  
where $1$ denotes the constant function taking its only value $1.$ Let
$\tau := \alpha \nu .$         Then 
$$
-\int _{J(\tilde{f})}\log \| \tilde{f}'\| d\tau = 
\frac{\partial P(t,f)}{\partial t}|_{t=\delta (f)}=-\int
_{J(\tilde{f})}\log \| \tilde{f}'\| d\mu .
$$  
Thus 
$$
\int _{J(\tilde{f})}\log \| \tilde{f}'\| d\tau =\int
_{J(\tilde{f})}\log \| \tilde{f}'\| d\mu .
$$ 
Since 
$$
\delta (f)=\frac{h_{\tau }(\tilde{f})}{\int _{J(\tilde{f})}\log
  \| \tilde{f}'\| d\tau } 
$$ 
(see \cite{sumi2}),  
it follows that $h_{\tau }(\tilde{f})=\log d.$ By the uniqueness of maximal entropy measure 
of $(\tilde{f}, J(\tilde{f}))$, 
we obtain that 
\begin{equation}
\label{eq:taumu}
\tau =\mu .
\end{equation}
Let $L_{\mu }:C(J(\tilde{f}))\rightarrow C(J(\tilde{f}))$ be the operator 
defined as follows 
$$
L_{\mu }(\varphi )(z)=\frac{1}{d}\sum _{\tilde{f}(y)=z}\varphi (y).
$$ 
Since $L_{\mu }^{\ast }(\mu )=\mu $, (\ref{eq:taumu}) implies that 
$L_{\mu }^{\ast }(\alpha \nu )=\alpha \nu .$ Thus, 
for any open subset $A$ of $J(\tilde{f})$ such that 
$\tilde{f}:A\rightarrow \tilde{f}(A)$ is injective, 
if $B$ is a Borel subset of $A$, then 
$(\alpha \nu )(\tilde{f}(B))=\int _{B}d\ d(\alpha \nu ).$ 
Moreover, we have 
\begin{align*}
(\alpha \nu )(\tilde{f}(B))
&= \int _{\tilde{f}(B)}\alpha d\nu \\ 
& =\int _{\tilde{f}(B)}(\alpha \circ \tilde{f})\circ (\tilde{f}|_{A})^{-1} d\nu \\ 
& =\int _{B}\alpha\circ \tilde{f}\ d((\tilde{f}|_{A}^{-1})_{\ast }\nu )\\ 
& =\int _{B}\alpha \circ \tilde{f}\cdot \frac{d((\tilde{f}|_{A}^{-1})_{\ast }\nu )}{d\nu } d\nu \\ 
& =\int _{B}(\alpha \circ \tilde{f})\cdot \| \tilde{f}'\|^{\delta (f)} d\nu .
\end{align*}
Thus $(\alpha \circ \tilde{f}(z))\cdot \| \tilde{f}'(z)\| ^{\delta (f)}=\alpha (z)d$ for 
$\nu $-a.e. $z\in J(\tilde{f}).$ 
Since supp$\, \tau =J(\tilde{f})$ (see \cite{sumi2}), 
it follows that 
\begin{equation}
\label{eq:atffdd1}
(\alpha \circ \tilde{f}(z))\cdot \| \tilde{f}'(z)\| ^{\delta (f)}=\alpha (z)d \ \ 
\mbox{ for every } z\in J(\tilde{f}). 
\end{equation}
Hence 
\begin{equation}
\label{eq:atffdd2}
\log \| \tilde{f}'(z)\| =\frac{1}{\delta (f)}(\log \alpha (z)-
\log \alpha \circ \tilde{f}(z)+\log d )\ \ 
\mbox{ for every } z\in J(\tilde{f}). 
\end{equation}
Therefore, for each $w\in \Sigma _{s}^{\ast }$ there exists a continuous 
function $\alpha _{w}: J(f_{w})\rightarrow \R $ such that 
\begin{equation}
\label{eq:atffddw}
\log \| f_{w}'(z)\| =\frac{1}{\delta (f)}(\log \alpha _{w}(z)-
\log \alpha _{w}\circ f_{w}(z)+|w|\log d )\ \ 
\mbox{ for every } z\in J(f_{w}). 
\end{equation}
Thus, for each $f_{w}$-invariant Borel probability measure $\beta $ on 
$J(f_{w})$, we have 
$$\int _{J(f_{w})}\log \| f_{w}'\| d\beta =|w|\frac{\log d}{\delta (f)}.$$ 
Let $p(t,w)$ be the topological pressure of $f_{w}:J(f_{w})\rightarrow J(f_{w})$ 
with respect to the potential function $-t\log \| f_{w}'\| .$ 
It follows that for each $w\in \Sigma _{s}^{\ast }$ with $\deg (f_{w})\geq 2$, 
\begin{equation}
\label{eq:ptwlin}
\frac{\partial p(t,w)}{\partial t} 
=-|w|\frac{\log d}{\delta (f)} \ \  \  \  \mbox{for each } t\in \R .
\end{equation}  
In particular, $t\mapsto p(t,w)$ is linear. 
Hence, 
$$\HD (J(f_{w}))=\frac{\log (\deg (f_{w}))}{\int _{J(f_{w})} \log \| f_{w}'\| d\mu _{w}},$$
where $\mu _{w}$ denotes the maximal entropy measure for $f_{w}: J(f_{w})\rightarrow J(f_{w}).$ 
Therefore, by Zdunik's theorem (\cite{zdunik2}),  
it follows that for each $w\in \Sigma _{s}^{\ast }$ with $\deg (f_{w})\geq 2$, there exists 
an $n_{w}\in \Z \setminus \{ 0,\pm 1\} $ and an element $\psi _{w}\in \mbox{Aut}(\oc )$ such that 
\begin{equation}
\label{eq:fwzd}
\psi _{w}\circ f_{w}\circ \psi _{w}^{-1}(z)=z^{n_{w}}\ \  \mbox{for every }z\in \oc .
\end{equation} 
In particular, there exists an element $\varphi \in \mbox{Aut}(\oc )$ such that 
$\varphi \circ f_{1}\circ \varphi ^{-1}(z)=z^{\pm d_{1}}$ for each $z\in \oc .$ 
Suppose that there exists a $j\in \{ 1,\ldots ,s\} $ such that 
$(\varphi \circ f_{j}\circ \varphi ^{-1})^{-1}(\{ 0,\infty \} )\neq \{ 0,\infty \} .$ 
If $d_{j}\geq 2$, then 
since each point of 
$(\varphi \circ f_{j}\circ \varphi ^{-1})^{-1}(\{ 0,\infty \} )$ is a critical point of 
$\varphi \circ f_{1}\circ f_{j}\circ \varphi ^{-1}$ and 
$\sharp (\varphi \circ f_{j}\circ \varphi ^{-1})^{-1}(\{ 0,\infty \} )\geq 3$,  
it contradicts (\ref{eq:fwzd}). 
If $d_{j}=1$, then since each point of 
$A:=(\varphi \circ f_{1}\circ \varphi ^{-1})^{-1}
((\varphi \circ f_{j}\circ \varphi ^{-1})^{-1}(\{ 0,\infty \} ))$ is a critical point of 
$\varphi \circ f_{1}\circ f_{j}\circ f_{1}\circ \varphi ^{-1}$ and $\sharp A\geq 3$, 
it contradicts (\ref{eq:fwzd}) again. 
Therefore, for each $j\in \{ 1,\ldots ,s\} $, 
\begin{equation}
\label{eq:vfv-1az}
\varphi \circ f_{j}\circ \varphi ^{-1}(z)=a_{j}z^{\pm d_{j}}
\end{equation}
 for some $a_{j}\in \C \setminus \{ 0\} .$ 
Since $G$ is expanding and $d_{1}\geq 2$, it follows that $d_{j}\geq 2$ for each $j=1,\ldots ,s.$ 
By (\ref{eq:ptwlin}) and (\ref{eq:vfv-1az}), 
 it follows that for each $j$, 
$$ \log d_{j}=\int _{J(f_{j})}\log \| f_{j}'\| d\mu _{j}=\frac{\log d}{\delta (f)}.$$ 
Therefore, $d_{1}=\cdots =d_{s}.$ 
Thus, we have completed the proof.   
\end{proof} 
Regarding Theorem~\ref{t:main1}, 
we give several remarks. 
In order to relate the Bowen parameter to the geometry of the Julia
set we need the concept of the  open set condition. We define it now.
\begin{dfn}
\label{d:osc}
Let $f=(f_{1},\ldots ,f_{s})\in (\mbox{{\em Rat}})^{s}$
and let $G=\langle f_{1},\ldots ,f_{s}\rangle $.  
Let also $U$ be a non-empty open set in $\oc .$ 
We say that $f$ (or $G$) satisfies the open set condition (with $U$) 
if 
$$
\cup _{j=1}^{s}f_{j}^{-1}(U)\subset U \  \text{  and } \
f_{i}^{-1}(U)\cap f_{j}^{-1}(U)=\emptyset 
$$ 
for each $(i,j)$ with $i\neq j.$ 
There is also a stronger condition. Namely, we say that 
$f$ (or $G$) satisfies the separating open set condition 
(with $U$) if 
$$
\cup _{j=1}^{s}f_{j}^{-1}(U)\subset U \  \text{  and } \
f_{i}^{-1}(\overline{U})\cap f_{j}^{-1}(\overline{U})=\emptyset 
$$ 
for each $(i,j)$ with $i\neq j.$ 
\end{dfn}
We remark that the above concept of ``open set condition'' 
(for ``backward IFS's'') is an analogue 
of the usual open set condition in the theory of IFS's. 

We introduce two other analytic invariants.  
\begin{dfn}[\cite{sumi2}]
Let $G$ be a countable rational semigroup. 
For any $t\geq 0$ and $z\in \oc $, we 
set 
$$S_{G}(z,t):=\sum _{g\in G}\sum _{g(y)=z}\| g'(y)\| ^{-t}
$$ 
counting multiplicities. We also set 
$$
S_{G}(z):= \inf \{ t\geq 0: S_{G}(z,t)<\infty \} 
$$ 
(if no $t$ exists with $S_{G}(z,t) <\infty $, then we set 
$S_{G}(z):=\infty $). Furthermore, 
we put 
$$
s_{0}(G):= \inf \{ S_{G}(z): z\in \oc \}
$$  
The number $s_{0}(G)$ is called the {\bf critical exponent of the 
Poincar\'{e} series} of $G.$ 
\end{dfn}
\begin{dfn}[\cite{sumi2}]
Let $f=(f_{1},\ldots ,f_{s})\in (\mbox{{\em Rat}})^{s}$, $t\geq 0$, and $z\in \oc .$
 We put
$$
T_{f}(z,t):=\sum _{\om \in \Sigma _{s}^{\ast}}
\sum _{f_{\om }(y)=z}\| f_{\om }'(y)\| ^{-t},
$$ 
counting multiplicities. Moreover, we set 
$$
T_{f}(z):=\inf \{ t\geq 0:T_{f}(z,t)<\infty \} 
$$ 
(if no $t$ exists with $T_{f}(z,t)<\infty $, then we set 
$T_{f}(z)=\infty $). 
Furthermore, we set 
$$
t_{0}(f):= \inf \{ T_{f}(z): z\in \oc \}.
$$  
The number $t_{0}(f)$ is called the {\bf critical exponent of 
the Poincar\'{e} series} of $f=(f_{1},\ldots ,f_{s})\in 
(\mbox{{\em Rat}})^{s}.$ 
\end{dfn}
\begin{rem}
\label{strem}
Let $f=(f_{1},\ldots ,f_{s})\in (\mbox{{\em Rat}})^{s}$, $t\geq 0$ , 
$z\in \oc $ and let 
$G=\langle f_{1},\ldots ,f_{s}\rangle .$ 
Then, 
$S_{G}(t,z)\leq T_{f}(t,z), S_{G}(z)\leq T_{f}(z),$ and 
$s_{0}(G)\leq t_{0}(f).$ Note that 
 for almost every $f\in (\mbox{{\em Rat}})^{s}$ with 
 respect to the Lebesgue measure, 
 $G=\langle f_{1},\ldots ,f_{s}\rangle $ is a free 
 semigroup and so we have 
 $S_{G}(t,z)=T_{f}(t,z), S_{G}(z)=T_{f}(z), $ and 
$s_{0}(G)=t_{0}(f).$  
\end{rem}

\begin{lem}[\cite{suetds1}]
\label{deltat0lem}
Let $f=(f_{1},\ldots ,f_{s})\in \Exp(s).$ 
Then $\delta (f)=t_{0}(f).$ 
\end{lem}
\begin{dfn}
Let $G$ be a rational semigroup. 
Then, we define
$$A(G):= \overline{\cup _{g\in G}g(\{ z\in \oc : \exists u\in G, 
u(z)=z, \| u'(z)\| <1\} )}.$$
\end{dfn}

Let us record the following fact proved in \cite{sumi2} . 
\begin{thm}[\cite{sumi2}]
\label{t:oscdhd}
Let $f=(f_{1},\ldots ,f_{s})\in \Exp(s)$ and 
let $G=\langle f_{1},\ldots ,f_{s}\rangle .$ 
Then, by \cite{sumi2} and Lemma~\ref{deltat0lem}, we have 
$\HD (J(G))\leq s_{0}(G)\leq S_{G}(z)\leq \delta (f)=T_{f}(z)=t_{0}(f),$ 
for each $z\in \oc \setminus (A(G)\cup P(G)).$ 
If in addition to the above assumption, 
$f$ satisfies the open set condition, then 
$$\HD(J(G))=s_{0}(G)=S_{G}(z)=\delta (f)=T_{f}(z)=t_{0}(f),$$ 
for each $z\in \oc \setminus (A(G)\cup P(G)).$ 
\end{thm}

In order to prove our second main theorem (see Theorem~\ref{t:main2}), 
we need some notation and lemmas from \cite{sumiprepare}. 
We shall provide the full proofs of these lemmas for the sake of
completeness of our exposition and convenience of the readers.  
\begin{dfn}
For each $s\in \N $, we set 
${\mathcal W}_{s}:=\{ (p_{1},\ldots ,p_{s})\in (0,1)^{s}\mid \sum _{j=1}^{s}p_{j}=1\} .$ 
\end{dfn}
\begin{dfn}[\cite{Se,J1,J2,sumiprepare}] 
\label{d:green} 
Let $f=(f_{1},\ldots , f_{s})\in {\mathcal P}^{s}.$ 
Let $\tilde{f}:\Sg_{s}\times \oc \rightarrow \Sg_{s}\times \oc $ be the skew product map
associated with $f.$ 
For any $\om \in \Sigma _{s}$, we set 
$$A_{\infty ,\om }:= \{ z\in \oc : f_{\om _{n}}\circ \cdots \circ f_{\om _{1}}(z)\rightarrow \infty 
\mbox{ as } n\rightarrow \infty\} .$$ 
For any $(\om ,y)\in \Sg_{s}\times \C $, 
let 
$$G_{\om }(y):= \lim _{n\rightarrow \infty }\frac{1}{\deg (f_{\om _{n}} \circ \cdots \circ f_{\om _{1}})}
\log ^{+}|f_{\om _{n}} \circ \cdots \circ f_{\om _{1}}(y)|,$$ 
where $\log ^{+}a:=\max\{\log a,0\} $ for each $a>0.$  
By the arguments in \cite{Se}, for each $\om \in \Sg_{s}$, the limit 
$G_{\om }(y) $ exists, the function
$G_{\om }$ is subharmonic on $\C $, and 
$G_{\om }|_{A_{\infty ,\om }}$ is equal to the Green's function on 
$A_{\infty ,\om }$ with pole at $\infty $.
Moreover, $(\om ,y)\mapsto G_{\om }(y)$ is continuous on $\Sg _{s}\times \C .$ 
Let $\mu _{\om }:=dd^{c}G_{\om }$, where $d^{c}:=\frac{i}{2\pi }(\overline{\partial }-\partial ).$ 
Note that by the argument in \cite{J1,J2}, 
$\mu _{\om }$ is a Borel probability measure on $J_{\om }(\tilde{f})$ such that 
$\mbox{supp}\, \mu _{\om }=J_{\om }(\tilde{f}).$ 
Furthermore, for each $\om \in \Sg_{s}$, 
let $\Omega (\om  )=\sum _{c} G_{\om }(c)$, where $c$ runs over all critical points of 
$f_{\om _{1}}$ in $\C $, counting multiplicities.   
\end{dfn}
\begin{rem}[\cite{hiroki1}]
\label{r:maxrelent}
Let $f=(f_{1},\ldots ,f_{s})\in (\mbox{{\em Rat}})^{s}.$ 
Let $\tilde{f}:\Sg_{s}\times \oc \rightarrow \Sg_{s}\times \oc $ 
be the skew product map associated with $f.$ 
Also, let $p=(p_{1},\ldots ,p_{s})\in {\mathcal W}_{s}$ and 
let $\tau $ be the Bernoulli measure on $\Sg_{s}$ with respect to the weight $p.$ 
Suppose that $\deg (f_{j})\geq 2$ for each $j=1,\ldots ,s.$  
Then, there exists a unique $\tilde{f}$-invariant Borel probability ergodic measure 
$\mu $ on $\Sg_{s}\times \oc $ such that $(\pi _{1})_{\ast }(\mu )= \tau $ and   
$$
h_{\mu }(\tilde{f}|\sigma )=\max _{\rho \in {\mathcal E}_{1}(\Sg_{s}\times \oc ): 
\tilde{f}_{\ast }(\rho )=\rho, (\pi _{1})_{\ast }(\rho )=\tau }  
h_{\rho }(\tilde{f}|\sigma )=\sum _{j=1}^{s}p_{j}\log (\deg (f_{j})),
$$
where $h_{\rho }(\tilde{f}|\sigma )$ denotes the relative metric entropy 
of $(\tilde{f},\rho )$ with respect to $(\sigma, \tau )$, and 
${\mathcal E}_{1}(\cdot )$ denotes the space of ergodic measures for $\tilde{f}:\Sigma _{s}\times \oc \rightarrow \Sigma _{s}\times \oc $ (see \cite{hiroki1}).  
The measure $\mu $ is called the {\bf maximal relative entropy
  measure} for $\tilde{f}$ with respect to  
$(\sigma , \tau ).$   
\end{rem}
\begin{lem}[\cite{sumiprepare}]
\label{l:maxentmeas1}
Let $f=(f_{1},\ldots ,f_{s})\in {\mathcal P}^{s}$ and  
let $G=\langle f_{1},\ldots ,f_{s}\rangle .$ 
Let $p=(p_{1},\ldots ,p_{s})\in {\mathcal W}_{s}.$  
Let $\tilde{f}:\Sg_{s}\times \oc \rightarrow \Sg_{s}\times \oc $ be the  
skew product associated with $f.$  
Let 
$\tau $ be the Bernoulli measure on $\Sg_{s}$ with respect to the weight $p.$ 
Let $\mu $  
be a Borel probability measure on $J(\tilde{f})$ defined by 
$$
\langle \mu ,\varphi \rangle := \int _{\Sg_{s}}
\int _{\oc  }\varphi (\om ,z) d\mu _{\om }(z)\, d\tau (\om )
$$ 
for any continuous function $\varphi $ on $\Sg_{s}\times \oc $, where 
$\mu _{\om }$ is the measure coming from Definition~\ref{d:green}. 
Then, $\mu $ is an $\tilde{f}$-invariant ergodic measure, 
$\pi _{\ast }(\mu )=\tau $,  and 
$\mu $ is the maximal relative entropy measure for $\tilde{f}$ with respect to $(\sigma , \tau )$ 
(see Remark~\ref{r:maxrelent}). 
\end{lem}
\begin{proof}
By the argument of the proof of \cite[Theorem 4.2(i)]{J2}, 
$\mu $ is $\tilde{f}$-invariant and ergodic, and 
$\pi _{\ast }(\mu )=\tau .$ 
Moreover, the argument of the proof of \cite[Theorem 5.2(i)]{J2},
yields that
$$
h_{\mu }(\tilde{f}|\sigma )\geq \int \log \deg (f_{\om  _{1}}) d\tau (\om )=
\sum _{j=1}^{m}p_{j}\log \deg (f_{j}).
$$ 
Combining this with 
\cite[Theorem 1.3(e)(f)]{hiroki1}, it follows that 
$\mu $ is the unique maximal relative entropy measure for $\tilde{f}$
with respect to $(\sigma, \tau ).$ 
\end{proof}
\begin{lem}[\cite{sumiprepare}] 
\label{l:lyapcalc}
Let $f=(f_{1},\ldots ,f_{s})\in {\mathcal P}^{s}$.  
Let $p=(p_{1},\ldots ,p_{s})\in {\mathcal W}_{s}.$  
Let 
$\tau $ be the Bernoulli measure on $\Sg_{s}$ with respect to the weight $p.$ 
Let $\tilde{f}:\Sg_{s}\times \oc \rightarrow \Sg_{s}\times \oc $ be the  
skew product associated with $f.$  
Let $\mu $ be the maximal relative entropy measure for $\tilde{f}$ with respect to 
$(\sigma , \tau ).$ 
Then 
$$
\int _{\Sg_{s}\times \oc }\log \|\tilde{f}'\|  d\mu =
\sum _{j=1}^{s}p_{j}\log \deg (f_{j})+\int _{\Sg_{s}}\Omega (\om  ) 
d\tau (\om ) .
$$ 
\end{lem}
\begin{proof}
For each $\om \in \Sg_{s}$, let 
$d(\om )=\deg (f_{\om _{1}})$ and 
$R(\om ):= \lim _{z\rightarrow \infty }(G_{\om }(z)-\log |z|).$ 
Also, we denote by $a(\om )$ the coefficient of the highest order term of $f_{\om _{1}}.$ 
Since $\frac{1}{d(\om )}G_{\sigma (\om )}(f_{\om _{1}}(z))=G_{\om }(z)$, 
we obtain that $R(\sigma (\om ))+\log |a(\om )|=d(\om )R(\om )$ for each 
$\om \in \Sg_{s}.$ 
Moreover, 
since $dd^{c}(\int _{\C }\log |w-z|d\mu _{\om }(w))=\mu _{\om }$ and 
$\int _{\C }\log |w-z|d\mu _{\om }(w)=\log |z|+o(1)$ as $z\rightarrow \infty $ (see \cite{Ra}), 
we have $\int _{\C }\log |w-z|d\mu _{\om }(w)=G_{\om }(z)-R(\om )$ for each $\om \in \Sg_{s}$ 
and $z\in \C .$ In particular, the function $\om \mapsto R(\om )$ is continuous on $\Sg_{s}.$ 
It follows from the above formula, that 
$$
\int _{\oc }\log |f_{\om _{1}}'(z)| d\mu _{\om }(z)=\log |a(\om
)|+\log d(\om )-(d(\om )-1)R(\om )+\Omega (\om )
$$  
for each $\om \in \Sg_{s}.$ 
In particular, the function $\om \mapsto \int _{\om }\log |f_{\om
  _{1}}'(z)| d\mu _{\om }(z)$ is continuous on  $\Sg_{s}.$ 
Furthermore, $\sigma _{\ast }(\tau )=\tau .$ 
From these arguments and Lemma~\ref{l:maxentmeas1}, 
we obtain 
\begin{align*}
\int _{\Sg_{s}\times \oc }\log |\tilde{f}'|d\mu 
& = \int _{\Sg_{s}}d\tau (\om )\int _{\oc }\log |f_{\om _{1}}'(z)| d\mu _{\om }(z)\\ 
&=\int _{\Sg_{s}}\left(\log |a(\om )|+\log d(\om )-(d(\om )-1)R(\om
  )+\Omega (\om )\right) d\tau (\om )\\  
&=\int _{\Sg_{s}}\left(R(\om )-R(\sigma (\om ))+\log d(\om )+\Omega (\om )\right) d\tau (\om )\\ 
&=\int _{\Sg_{s}}(\log d(\om )+\Omega (\om )) d\tau (\om )
=\sum _{j=1}^{s}p_{j}\log \deg (f_{j})+\int _{\Sg_{s}}\Omega (\om ) d\tau (\om ). 
\end{align*}
Moreover, since $\mu $ is $\tilde{f}$-invariant, and since the
Euclidian metric and the spherical metric  
are comparable on the compact subset $J(G)$ of $\C $, we have 
$\int _{\Sg_{s}\times \oc }\log |\tilde{f}'|d\mu= \int _{\Sg_{s}\times \oc }\log \|\tilde{f}'\| d\mu .$ 
Thus, we have proved our lemma. 
\end{proof}
\begin{lem}
\label{l:lyappbd}
Let $f=(f_{1},\ldots ,f_{s})\in {\mathcal P}^{s}.$ Let 
$d_{j}=\deg (f_{j})$ for each $j$ and let $d=\sum _{j}d_{j}.$ 
Let $\mu $ be the maximal entropy measure 
for $\tilde{f}:\Sigma _{s}\times \oc \rightarrow \Sigma _{s}\times \oc $ (see \cite{hiroki1}). 
Let $\tau $ be the Bernoulli measure on $\Sg _{s}$ with respect to the weight 
$(\frac{d_{1}}{d},\ldots ,\frac{d_{s}}{d}).$ 
Then, we have 
$$\int _{J(\tilde{f})}\log \| \tilde{f}'\| d\mu =\sum
_{j=1}^{s}\frac{d_{j}}{d}\log d_{j}+\int _{\Sg_{s}}\Omega (\om )d\tau
(\om ).$$ 
In particular, if, in addition to the assumptions of our lemma,
$\langle f_{1},\ldots ,f_{m}\rangle $ is postcritically bounded,  then  
$$\int _{J(\tilde{f})}\log \| \tilde{f}'\| d\mu =\sum _{j=1}^{s}\frac{d_{j}}{d}\log d_{j}.$$ 
\end{lem}
\begin{proof}
Let 
$$p=\left(\frac{d_{1}}{d},\ldots ,\frac{d_{s}}{d}\right)\in {\mathcal W}_{s}.$$  
Let $\tau $ be the Bernoulli measure on $\Sg_{s}$ with respect to the weight $p.$ 
By \cite{hiroki1}, $\mu $ is equal to the maximal relative entropy measure 
for $\tilde{f}$ with respect to $(\sigma , \tau ).$ 
By Lemma~\ref{l:lyapcalc}, the statement of our lemma holds.  
\end{proof}
We now give a lower estimate of the Hausdorff dimension of Julia sets
of  expanding finitely generated  
polynomial semigroups satisfying the open set condition. 
\begin{thm}
\label{p:hdlower}
Let $f=(f_{1},\ldots ,f_{s})\in \Exp(s)\cap {\mathcal
  P}^{s}$. Assume $f$ satisfies the open set condition.  
Let 
$d_{j}=\deg (f_{j})$ for each $j$ and let $d=\sum _{j}d_{j}.$ 
Let $\mu $ be the maximal entropy measure 
for $\tilde{f}:\Sigma _{s}\times \oc \rightarrow \Sigma _{s}\times \oc $ (see \cite{hiroki1}). 
Let $\tau $ be the Bernoulli measure on $\Sg _{s}$ with respect to the weight 
$(\frac{d_{1}}{d},\ldots ,\frac{d_{s}}{d}).$ 
Let $G=\langle f_{1},\ldots ,f_{s}\rangle .$   
Then, the following hold.
\begin{itemize}
\item[(1)]
\begin{equation}
\label{eq:hdjggeq}
\HD (J(G))=\delta (f)\geq \frac{\log d}{\int _{J(\tilde{f})}\log \| \tilde{f}'\| d\mu }
=\frac{\log d}{\sum _{j=1}^{s}\frac{d_{j}}{d}\log d_{j}+\int _{\Sg_{s}}\Omega (\om )d\tau (\om )}.
\end{equation}
\item[(2)]
If the inequality in {\em (\ref{eq:hdjggeq})} is replaced by the equality, then 

\begin{itemize}
\item[(a)]
$d_{1}=\cdots =d_{s}.$ We set $d_{0}=d_{1}=\cdots =d_{s}.$ 
\item[(b)] 
There exists an element $\varphi \in \mbox{{\em Aut}}(\C )$ and complex numbers $a_{1},\ldots , a_{s}$ with $a_{1}=1$ 
such that for each $j=1,\ldots ,s$, 
$\varphi f_{j}\varphi ^{-1}(z)=a_{j}z^{d_{0}}.$ 
\item[(c)] $\delta (f)=1+\frac{\log s}{\log d_{0}}.$ 
\end{itemize}
\item[(3)]
If, in addition to the assumptions of our lemma, $f\in \Epb(s)$,  
then 
$$\HD(J(G))=\delta (f)\geq \frac{\log d}{\sum _{j=1}^{s}\frac{d_{j}}{d}\log d_{j}}.$$ 

\end{itemize}
\end{thm}
\begin{proof}
By Theorem~\ref{t:main1}, Lemma~\ref{l:lyappbd} and Theorem~\ref{t:oscdhd}, 
we obtain the statement of our Theorem. 
\end{proof}
\begin{rem}
If $s>1$, then $\frac{\log d}{\sum _{j=1}^{s}\frac{d_{j}}{d}\log d_{j}}>1.$ 
\end{rem}
We now formulate and prove our second main theorem. 

\begin{thm}
\label{t:main2}
Let $f=(f_{1},\ldots ,f_{s})\in \Epb (s).$ 
Let 
$d_{j}=\deg (f_{j})$ for each $j$ and let $d=\sum _{j}d_{j}.$ 
Suppose that  
$(\log d )/(\sum _{j=1}^{s}\frac{d_{j}}{d}\log d_{j})\geq 2$ and 
$\delta (f)\leq 2.$ 
Then, we have the following.
\begin{itemize}
\item[(1)]
There exist a $\varphi \in \mbox{{\em Aut}}(\C )$ 
and non-zero complex numbers $a_{1},\ldots ,a_{s}$ 
such that for each $j=1,\ldots ,s$, 
$\varphi \circ f_{j}\circ \varphi ^{-1}(z)=a_{j}z^{s}$ for all $z\in \oc .$ 
\item[(2)] $d_{1}=\cdots =d_{s}=s$ and 
$$\delta (f)=2=\frac{\log d}{\sum _{j=1}^{s}\frac{d_{j}}{d}\log d_{j}}.$$ 
\end{itemize}
\end{thm}
\begin{proof}
By the assumptions of our theorem,
Theorem~\ref{t:main1} 
and Lemma~\ref{l:lyappbd}, 
we obtain 
$$2\leq \frac{\log d}{\sum _{j=1}^{s}\frac{d_{j}}{d}\log d_{j}}\leq \delta (f)\leq 2.$$ 
Therefore 
\begin{equation}
\label{eq:mth2pf1}
2=\frac{\log d}{\sum _{j=1}^{s}\frac{d_{j}}{d}\log d_{j}}= \delta (f). 
\end{equation}
Thus, by Lemma~\ref{l:lyappbd}, we obtain 
$$\frac{\log d}{\int _{J(\tilde{f})}\log \| \tilde{f}'\| d\mu }= \delta (f),$$ 
where $\mu $ denotes the maximal entropy measure for $\tilde{f}: \Sigma _{s}\times \oc \rightarrow 
\Sigma _{s}\times \oc .$ By Theorem~\ref{t:main1}, 
it follows that there exists a $\varphi \in \mbox{Aut}(\C )$,  
non-zero complex numbers $a_{1},\ldots ,a_{s}$, and a number $d_{0}\in \N $ such that 
$d_{0}=d_{1}=\cdots d_{s}$ and $\varphi \circ f_{j}\circ \varphi ^{-1}(z)=a_{j}z^{d_{0}}$ for all $z\in \oc .$ 
By (\ref{eq:mth2pf1}), we obtain 
$$2=\frac{\log d}{\sum _{j=1}^{s}\frac{d_{j}}{d}\log d_{j}}=1+\frac{\log s}{\log d_{0}}.$$ 
Therefore, $d_{0}=s.$ Thus, we have completed the proof.  
\end{proof}
\begin{rem}
\label{r:deltal2}
Let $f=(f_{1},\ldots ,f_{s})\in \Exp (s).$ Suppose that 
$f$ satisfies the open set condition. Then $\delta (f)=\HD (J(\langle f_{1},\ldots ,f_{s}\rangle ))\leq 2$ (see \cite{sumi2}, \cite{sumi06}). 
\end{rem}
\begin{cor}
\label{c:2cri}
Let $f=(f_{1},f_{2})\in \Epb (2) .$ Suppose that $\deg (f_{1})=\deg (f_{2})=2.$ 
Then, the following statements {\em (1),(2), (3), (4)} are equivalent.
\begin{itemize}
\item[(1)] $\delta (f)\leq 2.$ 
\item[(2)] $\delta (f)= 2.$ 
\item[(3)] There exists a $\varphi \in \mbox{{\em Aut}}(\C )$ and a  
non-zero complex number $a$ such that 
$$\varphi \circ f_{1}\circ \varphi ^{-1}(z)= z^{2}, \varphi \circ f_{2}\circ \varphi ^{-1}(z)=
az^{2}\mbox{ for all } z\in \oc .$$ 
\item[(4)] either 
\begin{itemize}
\item[(a)] $f$ satisfies the open set condition or 
\item[(b)] there exists a $\varphi \in \mbox{{\em Aut}}(\C )$ and a  
complex number $a$ with $|a|=1$ such that 
$$\varphi \circ f_{1}\circ \varphi ^{-1}(z)= z^{2}, \varphi \circ f_{2}\circ \varphi ^{-1}(z)=
az^{2} \mbox{ for all } z\in \oc .$$  
\end{itemize}
\end{itemize} 
\end{cor}
\begin{proof}
``(1)$\Rightarrow $(2)'' and   
``(2)$\Rightarrow $(3)'' follow from  Theorem~\ref{t:main2}. 
It is easy to see ``(3)$\Rightarrow $(4)''.  
``(4)$\Rightarrow $(1)'' follows from Remark~\ref{r:deltal2}. 
Thus, we have completed the proof.   
\end{proof}
\section{Remarks and examples}
\label{Remarks} 
In this section we collect some remarks and construct relevant
examples illustrating our main theorems. 

\begin{rem}[\cite{sumid3,sumid1}] 
\label{shshfinprop}
Let $s\geq 2$ and let $d_{2},\ldots ,d_{s}\in \N $ be such that  
$d_{j}\geq 2$ for each $j=2,\ldots ,s.$ Let 
$f_{1}\in \Epb(1).$ 
Let $b_{2},b_{3},\ldots ,b_{s}\in $ {\em int}$(K(f_{1})).$ 
Then, the following statements hold. 
\begin{enumerate}
\item \label{shshfinprop1}
There exists a number $c>0$ such that 
for each $(a_{2},a_{3},\ldots ,a_{s})\in \C ^{s-1}$ with 
$0<|a_{j}|<c$ ($j=2,\ldots ,s$), 
setting $f_{j}(z)=a_{j}(z-b_{j})^{d_{j}}+b_{j}$ ($j=2,\ldots ,s$), 
we have $(f_{1},\ldots ,f_{s})\in \Epb(s)$.
\item \label{shshfinprop2}
 Suppose also that 
either (i) there exists a $j\geq 2$ with $d_{j}\geq 3$, or 
(ii) $\deg(f_{1})=3$, $b_{2}=\cdots =b_{s}.$ Then, there exist 
$a_{2},a_{3},\ldots ,a_{s}>0$ such that setting 
$f_{j}(z)=a_{j}(z-b_{j})^{d_{j}}+b_{j}$ ($j=2,\ldots ,s$), 
we have $(f_{1},\ldots ,f_{s})\in \Epb(s)$ and 
$J(\langle f_{1},\ldots ,f_{s}\rangle )$ is disconnected.  
\end{enumerate} 
In \cite{sumid1,sumid3}, the first author of this paper 
provided a lot of methods of constructing of examples of elements of $\Epb(s). $ 
\end{rem}

We give below concrete examples of expanding polynomial semigroups
satisfying the open set condition.  
 
\begin{rem}
\label{r:dless2}
Let $f_{1}\in \Epb(1)$ and let $b\in \mbox{{\em int}}(K(f_{1})).$ 
Let $d_{1}:=\deg (f_{1}).$  
Let $d_{2}\in \N $ with $d_{2}\geq 2$ and suppose that 
$(d_{1}, d_{2})\neq (2,2).$ 
Then there exists a number $c>0$ such that for each 
$a\in \C $ with $0<|a|<c$, 
setting $f_{2}(z):=a(z-b)^{d_{2}}+b$ and $f=(f_{1}, f_{2})\in {\mathcal P}^{2}$, 
we have 
\begin{itemize}
\item[(a)]
$f\in \Epb(2)$,  
\item[(b)] $f$ satisfies the separating open set condition,  
\item[(c)] $\delta (f)<2$, and  
\item[(d)] setting $G=\langle f_{1},f_{2}\rangle $, 
we have $\HD (J(G))=\delta (f)<2.$ 
\end{itemize}
  For the proof of this result, see \cite{suetds1}. 
Moreover, by Theorem~\ref{p:hdlower}, setting $d:=d_{1}+d_{2}$, 
we have 
$$\HD (J(G))\geq \frac{\log d}{\sum _{j=1}^{2}\frac{d_{j}}{d}\log d_{j}}>1.$$ 
If $f_{1}$ and $f_{2}$ are not simultaneously conjugate to the form 
$az^{2}$ by an element in $\mbox{{\em Aut}}(\C )$, then by Theorem~\ref{t:main1}, 
$$\HD (J(G))> \frac{\log d}{\sum _{j=1}^{2}\frac{d_{j}}{d}\log d_{j}}>1.$$ 
See also Figure~\ref{fig:dcgraph}. 
\begin{figure}[htbp]
\caption{The Julia set of $G=\langle g_{1}^{2},g_{2}^{2}\rangle $, 
where $g_{1}(z):= z^{2}-1, g_{2}(z):= \frac{z^{2}}{4}.$ 
$f:=(g_{1}^{2},g_{2}^{2})$ satisfies (a)--(d) in Remark~\ref{r:dless2}. 
Moreover, by Theorem~\ref{p:hdlower}, 
$\frac{\log 8}{\log 4}=\frac{3}{2}<\HD(J(G))<2.$
}
\includegraphics[width=3cm,width=3cm]{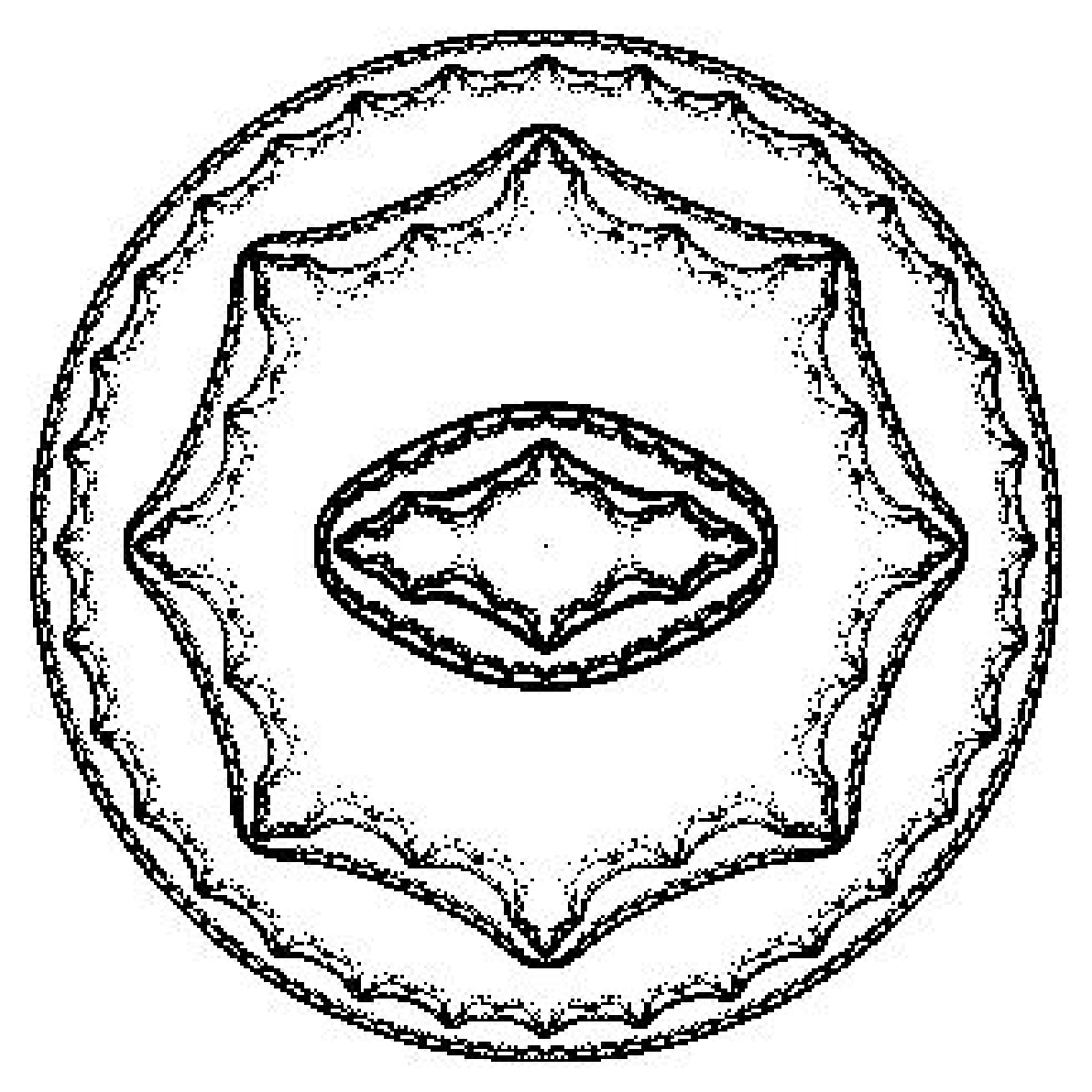}
\label{fig:dcgraph}
\end{figure}

\end{rem}
We give  examples of elements $f=(f_{1},f_{2})\in \Epb(2)$ with $\delta (f)>2.$ 
\begin{ex}[\cite{sumid3}]
\label{ex:dlarger2}
Let $f_{1}\in \Epb(1)$ with $\deg (f_{1})=2.$ 
Let $b\in \mbox{{\em int}}(K(f_{1}))$, where $K(\cdot )$ denotes the filled-in Julia set. 
Then, by \cite{sumid3}, there exists a number $c>0$ such that 
for each $a\in \C $ with $0<|a|<c$, setting $f_{2}(z)=a(z-b)^{2}+b$, 
we have $f:=(f_{1},f_{2})\in \Epb(2).$ By Corollary~\ref{c:2cri}, 
it follows that if $f_{1}$ and $f_{2}$ are not simultaneously conjugate to the form $az^{2}$ by 
an element in $\mbox{{\em Aut}}(\C )$, 
then $\delta (f)>2.$ See Figure~\ref{fig:dlarger2}. 
  \begin{figure}[htbp]
\caption{The Julia set of $G=\langle f_{1},f_{2}\rangle $, 
where $f_{1}(z):= z^{2}-1, f_{2}(z):= 0.09z^{2}.$ 
$f:=(f_{1},f_{2})$ belongs to $\Epb(2) $ (see \cite{sumid3}). 
By Corollary~\ref{c:2cri},  $\delta (f)>2.$ }
\includegraphics[width=3cm,width=3cm]{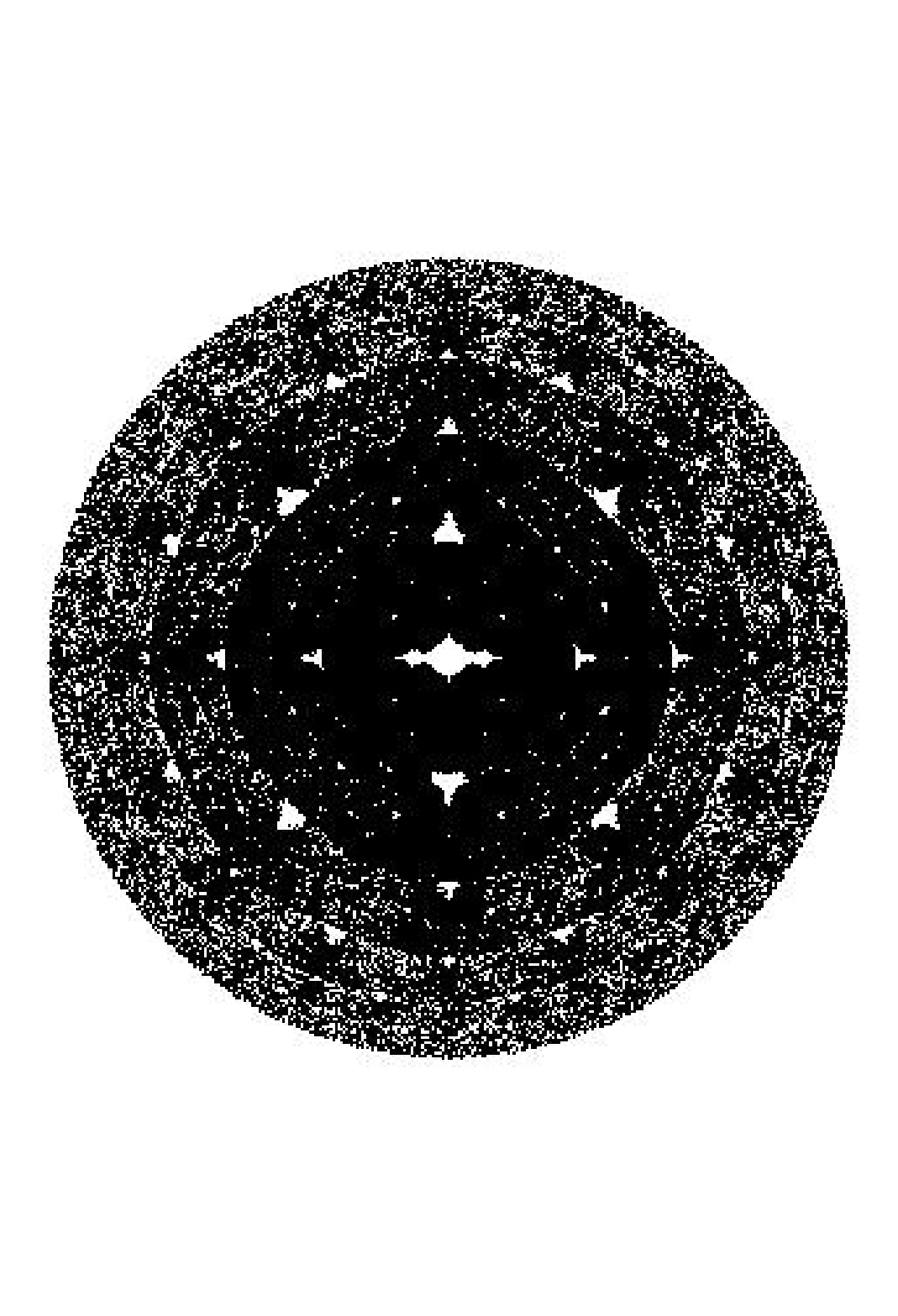}
\label{fig:dlarger2}
\end{figure}
 
\end{ex}
\begin{ex}
Let $f_{1}(z)=z^{2}$. For each $c\in \C $, let $f_{2,c}=\frac{1}{4}z^{2}+c.$ 
Let $f_{c}:=(f_{1},f_{2,c})$ and $G_{c}:= \langle f_{1},f_{2,c}\rangle .$  
Then by Lemma~\ref{l:epbopen}, 
there exists a number $c_{0}>0$ such that for each $c\in \C $ with $|c|<c_{0}$, 
$f_{c}\in \Epb(2).$ Moreover, by Corollary~\ref{c:2cri}, $\delta (f_{0})=2$ and 
for each $c\in \C $ with $0<|c|<c_{0}$, $\delta (f_{c})>2.$ 
Let $U_{0}$ be the connected component of $\Exp (2)$ with $(f_{1},f_{2,0})\in U_{0}.$ 
Since $c\mapsto \delta (f_{c})$ is real analytic on $U_{0}$ (\cite{suetds1}), it follows that 
$c\mapsto \delta (f_{c})$ is not constant in any open subset of $U_{0}.$ 
\end{ex}
\begin{rem}
\label{r:dimran}
 We remark that, as illustrated in \cite{sumikokyuroku, sumiprepare}, 
 estimating the Hausdorff dimension of the Julia sets of 
 rational semigroups plays an important role when we 
 investigate random complex dynamics and its associated 
 Markov process on $\oc .$ 
For example, 
 when we consider the random dynamics of a compact 
 family $\Gamma $ of polynomials 
 of degree greater than or equal to two,  
 then the function $T_{\infty }:\oc \rightarrow 
 [0,1]$ representing the probability of tending to $\infty \in \oc $ 
 varies only on a subset of the Julia set of the polynomial semigroup  
 generated by $\Gamma $, and under certain conditions, 
 the function $T_{\infty }:\oc \rightarrow [0,1]$ is continuous on $\oc .$
If the Hausdorff dimension of the Julia set is strictly less than two, 
then it means that $T_{\infty }:\oc \rightarrow [0,1]$ is a 
complex version of devil's staircase (Cantor function) 
(\cite{sumirandom,sumiprepare}). 
For example, setting $g_{1}(z):=z^{2}-1, g_{2}(z):=\frac{z^{2}}{4}$, 
 $f_{1}:=g_{1}^{2}$, and $f_{2}:=g_{2}^{2}$,
 we consider the random dynamics on $\oc $ such that at every step we 
 choose a map $f_{j}$ with probability $0<p_{j}<1$, where $p_{1}+p_{2}=1.$ 
Then the function $T_{\infty }$ representing the probability of tending to $\infty $ 
is continuous on $\oc $ and varies exactly on the Julia set (Figure~\ref{fig:dcgraph}) 
of the polynomial semigroup $\langle f_{1},f_{2}\rangle $, whose
Hausdorff dimension is strictly less than two (see
\cite{sumirandom,sumiprepare}).   
\end{rem}

\end{document}